\documentclass[10pt]{amsart}

\usepackage{amsfonts,amssymb,amsmath,amscd,amstext}
\usepackage[colorlinks=true,linkcolor=blue,citecolor=blue]{hyperref}
\usepackage[utf8]{inputenc}
\usepackage{microtype}
\usepackage{graphicx}
\usepackage{changes}
\usepackage{comment}
\usepackage{lmodern}
\usepackage[a4paper,margin=3.5cm]{geometry}
\usepackage{tikz}
\usepackage{float}

\renewcommand{\leq}{\leqslant}
\renewcommand{\geq}{\geqslant}
\renewcommand{\le}{\leqslant}
\renewcommand{\ge}{\geqslant}
\newcommand{\ptl}{\partial}
\newcommand{\hhh}{{\mathcal{H}}}
\newcommand{\norm}[1]{|| #1 ||}

\newcommand{\rr}{{\mathbb{R}}}

\newcommand{\hh}{{\mathbb{H}}}

\newcommand{\sph}{{\mathbb{S}}}
\newcommand{\la}{\lambda}

\newcommand{\Om}{\Omega}

\newcommand{\escpr}[1]{\langle#1\rangle}

\newcommand{\mh}{\mathcal{H}}

\definecolor{asparagus}{rgb}{0.53, 0.66, 0.42}

\DeclareMathOperator{\divv}{div}
\DeclareMathOperator{\intt}{int}

\newtheorem{theorem}{Theorem}[section]
\newtheorem{proposition}[theorem]{Proposition}
\newtheorem{lemma}[theorem]{Lemma}
\newtheorem{corollary}[theorem]{Corollary}

\theoremstyle{definition}

\newtheorem{remark}[theorem]{Remark}
\newtheorem{example}[theorem]{Example}

\newtheorem{definition}[theorem]{Definition} 

\theoremstyle{remark}

\numberwithin{equation}{section}

\begin{document}

\title[Area-minimizing graphs in the sub-Finsler Heisenberg group $\hh^1$]{Area-minimizing horizontal graphs with low-regularity in the sub-Finsler Heisenberg group $\hh^1$}

\author[G.~Giovannardi]{Gianmarco Giovannardi}
\address{Dipartimento di Matematica Informatica "U. Dini", Università degli Studi di Firenze, Viale Morgani 67/A, 50134, Firenze, Italy}
\email{gianmarco.giovannardi@unifi.it}

\author[J.~Pozuelo]{Juli\'an Pozuelo} 
\address{Departamento de Geometría y Topología \& Research Unit MNat \\
Universidad de Granada \\ E--18071 Granada \\ Spain}
\email{pozuelo@ugr.es}

\author[M.~Ritoré]{Manuel Ritoré} 
\address{Departamento de Geometría y Topología \& Research Unit MNat \\
Universidad de Granada \\ E--18071 Granada \\ Spain}
\email{ritore@ugr.es}

\date{\today}
\maketitle

\begin{abstract}
In the Heisenberg group $\mathbb{H}^1$, equipped with a left-invariant and not necessarily symmetric norm in the horizontal distribution,  we provide examples of entire area-minimizing horizontal graphs which are locally Lipschitz in Euclidean sense. A large number of them fail to have further regularity properties. The examples are obtained  by prescribing as singular set a horizontal line or a finite union of horizontal half-lines extending from a given point. We also provide examples of families of area-minimizing cones.
\end{abstract}

\section{Introduction}
\label{sc:introduction}

The regularity of perimeter-minimizing sets in sub-Finsler geometry is currently one of the most challenging problems in Calculus of Variations. A sub-Finsler structure in a Carnot-Carathéodory manifold with a completely non-integrable distribution $\hhh$ is defined by a smooth norm on $\hhh$. The case of a Euclidean norm is that of sub-Riemannian geometry. The notion of sub-Riemannian perimeter was introduced by Garofalo and Nhieu \cite{MR1404326}, while sub-Finsler boundary measures in the first Heisenberg group $\hh^1$ were considered by Sánchez \cite{snchez2017subfinsler}, and sub-Finsler perimeters in $\hh^1$ by Pozuelo and Ritoré \cite{PozueloRitore2021} and Franceschi et al.~\cite{monti-finsler}.

Fine properties of sets of finite perimeter in the Heisenberg groups $\hh^n$ were obtained by Franchi, Serapioni and Serra-Cassano \cite{MR1871966}. Among others, they obtained a structure result for the reduced boundary of a set of finite perimeter: except for a set of small spherical Hausdorff dimension, it is the union of $\hh$-regular hypersurfaces (i.e., level sets of continuous functions with continuous first derivatives in the horizontal directions), see the Main Theorem in page~486 of \cite{MR1871966}. 

The regularity of sub-Riemannian perimeter-minimizing sets has been investiga\-ted by a large number of researchers \cite{MR2165405,MR2435652,MR2405158,MR2333095,MR2648078,MR2609016,MR3044134,MR2875642,MR3259763,MR2448649,MR2481053,MR2262784,MR3984100,MR2583494}. The boundaries of the conjectured solutions to the isoperimetric problem are of class $C^2$, see \cite{MR2312336}, although there exist examples of area-minimizing horizontal graphs which are merely Euclidean Lipschitz, see \cite{MR2262784,MR2455341, MR2448649}. The sub-Riemannian Plateau problem was first  considered by Pauls \cite{MR2043961}. Afterwards, under given Dirichlet conditions on $p$-convex domains, Cheng, Hwang and Yang \cite{MR2262784} proved existence and uniqueness of $t$-graphs (horizontal graphs of the form $t=u(x,y)$) which are Lipschitz continuous weak solutions of the minimal surface equation in $\hh^1$. Later, Pinamonti, Serra Cassano, Treu and Vittone  \cite{MR3445204} obtained existence and uniqueness of $t$-graphs on domains with boundary data satisfying a bounded slope condition, thus showing that Lipschitz regularity is optimal at least in the first Heisenberg group $\hh^1$. Capogna, Citti and Manfredini \cite{MR2583494} established that the intrinsic graph of a Lipschitz continuous function which is, in addition, a viscosity solution of the sub-Riemannian minimal surface equation in $\hh^1$, is of class $C^{1,\alpha}$, with higher regularity in the case of~$\hh^n$, $n>1$, see \cite{MR2774306}. It was shown in \cite{MR2481053} that the regular part of a $t$-graph of class $C^1$ with~continuous  prescribed sub-Riemannian mean curvature in $\hh^1$ is foliated by $C^2$ characteristic curves. Furthermore, in \cite{MR3412382} the authors generalized the previous result when the boundary $S$ is a general $C^1$ surface in a three-dimensional contact sub-Riemannian manifold. Later, Galli in \cite{MR3474402} improved the result in \cite{MR3412382} only assuming that the boundary $S$ is Euclidean Lipschitz and $\hh$-regular. Recently, in \cite{MR4314055} the first and third authors extended the result in \cite{MR3474402} to the sub-Finsler Heisenberg groups. Up to now, determining the optimal regularity of perimeter-minimizing $\hh$-regular hypersurfaces in the Heisenberg group remains an open problem.

Bernstein type problems for surfaces in $\hh^1$ have also received a special attention. The nature of the sub-Riemannian Bernstein problem in the Heisenberg group  is completely different from the Euclidean one even for graphs.  On the one hand the area functional for $t$-graphs is convex as in the Euclidean setting. Therefore the critical points of the area are automatically minimizers for the area functional. However, since $t$-graphs admit singular points where the horizontal gradient vanishes their classification is not an easy task. Thanks to a deep study of the singular set for $C^2$ surfaces in $\hh^1$,  Cheng, Hwang, Malchiodi, and  Yang   \cite{MR2165405} showed that minimal $t$-graphs of class $C^2$ are congruent to a family of surfaces including the hyperbolic paraboloid $u(x,y)=xy$ and the Euclidean planes. Under the hypothesis that the surface is area-stationary, Ritoré and Rosales  proved in \cite{MR2435652} that the surface must be congruent to a hyperbolic paraboloid or to a Euclidean plane. If we consider the class of Euclidean Lipschitz $t$-graphs, the previous classification does not hold since there are several examples of area-minimizing surfaces of low regularity, see \cite{MR2448649}. The complete classification for $C^2$ surfaces was established by Hurtado, Ritoré and Rosales in \cite{MR2609016}, by showing that a complete,  orientable,  connected,  stable  area-stationary surface is congruent either to the hyperbolic paraboloid $u(x,y)=xy$ or to a Euclidean plane. As in the Euclidean setting the stability condition is crucial in order to discard some minimal surfaces such as helicoids and catenoids.

On the other hand, the study of the regularity of intrinsic graphs (i.~e., Riemannian graphs over vertical planes) is a completely different problem since the area functional for such graphs is not convex. Indeed, 
 Danielli, Garofalo,  Nhieu in \cite{MR2405158} discovered that  the  family of graphs
 \[
 u_{\alpha}(x,t)=\frac{\alpha x t}{1+2\alpha x^2},\quad\alpha>0,
 \]
 are area-stationary but \emph{unstable}. In \cite{MR2455341}, Monti, Serra Cassano and Vittone provided an example of an area-minimizing intrinsic graph of regularity $C^{1/2}(\rr^2)$ that is an intrinsic cone. Therefore the Euclidean threshold of dimension $n=8$ fails in the sub-Riemannian setting. In \cite{MR2333095}, Barone Adesi, Serra Cassano and Vittone classified complete $C^2$ area-stationary intrinsic graphs. Later    Danielli, Garofalo, Nhieu and Pauls in \cite{MR2648078} showed that a $C^2$ complete stable embedded minimal surface in $\hh^1$ with empty characteristic set must be a plane.
In \cite{MR3406514} Galli and Ritoré proved that a  complete, oriented  and stable area-stationary $C^1$ surface without singular points is a vertical plane.
Later,  Nicolussi Golo and  Serra Cassano \cite{MR3984100}  showed that Euclidean Lipschitz stable area-stationary intrinsic graphs are vertical planes. Recently, Giovannardi and Ritoré \cite{2021arXiv210502179G} showed that in the Heisenberg group $\hh^1$ with a sub-Finsler~structure, a complete, stable, Euclidean Lipschitz surface without singular points is a vertical plane and Young \cite{2021arXiv210508890Y} proved that a ruled area-minimizing entire intrinsic graph in $\hh^1$ is a vertical plane by introducing a family of deformations of graphical strips based on variations of a vertical curve.

In this note, we provide examples of entire perimeter-minimizing $t$-graphs for a fixed but arbitrary left-invariant  sub-Finsler structure in the first Heisenberg group $\hh^1$. Our examples are inspired by the corresponding sub-Riemannian ones in \cite{MR2448649}. Of particular interest are the conical examples,  invariant by the non-isotropic dilations of $\hh^1$. In the sub-Riemannian case these examples were investigated in \cite{MR4307010} and \cite{MR2448649}.

The paper is organized the following way. In Section~\ref{sc:preliminaries} we include some preliminaries. In Theorem~\ref{thm:1stvar} of Section~\ref{sec:1st} we obtain a necessary and sufficient condition, inspired by Theorem~3.1 in \cite{PozueloRitore2021}, for a surface to be a critical point of the sub-Finsler area. We assume that the surface is piecewise $C^2$, and composed of pieces meeting in a $C^1$ way along $C^1$ curves. This condition will allow us to construct area-minimizing examples in Proposition~\ref{prop:one_singular_line} of Section~\ref{sec:one_singular_line}, and examples with low regularity in Proposition~\ref{prop:cont}. The same construction, keeping fixed the angle at one side (and hence at the other one) of the singular line, provides examples of area-minimizing cones, see Corollary~\ref{cor:cones}. Finally, in Section~\ref{sec:cones} we exhibit some examples of area-minimizing cones in the spirit of \cite{MR4307010}. These examples are obtained in Theorem~\ref{thm:cones} from circular sectors of the area-minimizing cones with one singular half-line obtained in Corollary~\ref{cor:cones}.

\section{Preliminaries}
\label{sc:preliminaries}
\subsection{The Heisenberg group}
\label{sc:heis}
We denote by $\mathbb{H}^1$ the first Heisenberg group: the $3$-dimensional Euclidean space $\rr^3$ with coordinates $(x,y,t)$, endowed with the product $*$ defined by 
\[
(x,y,t)*(\bar{x},\bar{y} ,\bar{t})=(x+\bar{x}, y+\bar{y},t+\bar{t}+ \bar{x}y-x\bar{y}).
\]
A frame of left-invariant vector fields is given by 
\[
X=\dfrac{\partial}{\partial x} + y \dfrac{\partial}{\partial t}, \qquad Y=\dfrac{\partial}{\partial y} - x \dfrac{\partial}{\partial t}, \qquad T=\dfrac{\partial}{\partial t}.
\]
For $p \in \hh^1$, the left translation by $p$ is the diffeomorphism $L_p(q) =p*q$.
The horizontal distribution $\mathcal{H}$ is the planar non-integrable one generated by $X$ and $Y$, which coincides with the kernel of the contact one-form $\omega=dt-y dx+x dy$.

We shall consider on $\mathbb{H}^1$ the left-invariant Riemannian metric $g= \escpr{\cdot,\cdot}$, so that $\{X, Y, T\}$ is a global orthonormal frame, and let   $D$ be the Levi-Civita connection associated to the Riemannian metric $g$. 
Setting $J(U)=D_U T$ for any vector field $U$ in $\hh^1$ we get $J(X)=Y$, $J(Y)=-X$ and $J(T)=0$. Therefore $-J^2$ coincides with the identity when restricted to the horizontal distribution.
 The Riemannian volume of a set $E$ is, up to a constant, the Haar measure of the group and is denoted by $|E|$. The integral of a function $f$ with respect to the Riemannian measure is denoted by $\int f\, d\hh^1$.
 
\subsection{Sub-Finsler norms and perimeter}
\label{sc:subfinsler}
Given a convex set $K\subset\rr^2$ with $0\in\intt(K)$ and associated asymmetric norm $\norm{\cdot}$ in $\rr^2$, we define on $\hh^1$ a left-invariant norm $\norm{\cdot}_K$ on the horizontal distribution by means of the equality
\[
\norm{fX+gY}_K(p)=\norm{(f(p),g(p))},
\] 
for any $p\in\hh^1$. Its dual norm is denoted by $\norm{\cdot}_{K,*}$. 

If  the boundary of $K$ is of class $C^\ell$, for $\ell\ge 2$, and the geodesic curvature of $\ptl K$ is strictly positive, we say that $K$ is of class $C^\ell_+$. When $K$ is of class $C^2_+$, the outer Gauss map $N_K:\ptl K\to\sph^1$ is a diffeomorphism and the map
\[
\pi_K(fX+gY)=N_K^{-1}\bigg(\frac{(f,g)}{\sqrt{f^2+g^2}}\bigg),
\]
defined for non-vanishing horizontal vector fields $U=fX+gY$, satisfies
\[
\norm{U}_{K,*}=\escpr{U,\pi_K(U)},
\]
where $||\cdot||_{K,*}$ is the dual norm of $||\cdot||_K$. See \S 2.3 in \cite{PozueloRitore2021}.

\begin{definition}
Given a convex body $K\subset\rr^2$ containing $0$ in its interior, and a measurable set $E\subset \hh^1$, its horizontal $K$-perimeter in an open set $\Om\subset\hh^1$ is
\[
P_{K}(E,\Om)= \sup \left\{ \int_E \divv (U) \ d\hh^1, U \in \mathcal{H}^1_0(\Om), \norm{U}_{K, \infty} \le 1 \right\},
\]
Here $\norm{U}_{K, \infty}=\sup_{p \in \hh^1} \norm{U_p}_{K_0}$ and $\mathcal{H}_0^1(\Om)$ is the space of $C^1$ horizontal vector fields with compact support in $\Om$. If $\Om=\hh^1$ we write $P_K(E)$ instead of $P_K(E,\hh^1)$. When $P_{K}(E,\Omega)$ is finite we say that $E$ has finite horizontal~$K$-perimeter in $\Omega$.
\end{definition}

\begin{remark}
If $E$ has $C^1$ boundary $\partial E$, then its perimeter $P_K(E)$ is equal to the sub-Finsler area $A_K$ of its boundary, defined by
\[
A_{K}(\partial E)=\int_{\partial E} \norm{N_h}_{K,*} d\sigma.
\]
where $N_h$ is the projection on the horizontal distribution $\mh$ of the Riemannian normal $N$ with respect to the metric $g$, and $d\sigma$ is the Riemannian measure of $\partial E$. For more details see \S 2.4 in \cite{PozueloRitore2021}.
\end{remark}

We will often omit the subscript $K$ to simplify the notation.

\section{The first variation formula and a stationary condition}
\label{sec:1st}
In this section we present some consequences of the first variation formula. We assume that the Heisenberg group $\hh^1$ is endowed with the sub-Finsler structure associated to a convex set $K$ of class $C^2_+$ with $0\in\intt(K)$. Recall that, given a surface $S\subset\hh^1$ of class $C^1$, its \emph{singular set} $S_0$ is composed of those points of $S$ where the~tangent plane is horizontal. The \emph{regular} part of $S$ is $S\smallsetminus S_0$.

\begin{theorem}[{Theorem~3.1 in \cite{PozueloRitore2021}}]
\label{thm:1stvar}
\noindent Let $S$ be an oriented surface of class $C^1$ such that the regular part $S\smallsetminus S_0$ is  of class $C^2$. Consider a $C^2$ vector field $U$ with compact support on $S$, normal component $u=\escpr{U,N}$, and associated flow $\{\varphi_s\}_{s\in\rr}$. Let $\eta=\pi(\nu_h)$, where $\nu_h$ is the horizontal unit normal to $S$. Then we have
\begin{equation}
\label{eq:1stvar}
\frac{d}{ds}\bigg|_{s=0} A_K(\varphi_s(S))=\int_{S\setminus S_0} H_Ku\,dS-\int_{S\setminus S_0} \divv_S(u\eta^\top)\,dS,
\end{equation}
where $\divv_S$ is the Riemannian divergence on $S$ and the superscript $\top$ indicates the projection over the tangent plane to $S$. The quantity $H_K=\escpr{\nabla_Z\pi(\nu_h),Z}$, for $Z=-J(\nu_h)$, is the $K$-mean curvature of $S$.
\end{theorem}

Using Theorem~\ref{thm:1stvar} we can prove the following necessary and sufficient condition for a surface $S$ to be $A_K$-stationary. When a surface $S$ of class $C^1$ is divided into two parts $S^+,S^-$ by a singular curve $S_0$ so that $S^+,S^-$ are of class $C^2$ up to the boundary, the tangent vectors $Z^+,Z^-$ can be chosen so that they parameterize the characteristic curves (i.~e., horizontal curves en the regular part of $S$) as curves leaving from $S_0$, see Corollary~3.6 in \cite{MR2165405} . In this case $\eta^+=\pi(\nu_h)=\pi(J(Z^+))$ and $\eta^-=\pi(J(Z^-))$.

\begin{corollary}
\label{cor:minimal}
Let $S$ be an oriented surface of class $C^1$ such that the singular set $S_0$ is a $C^1$ curve. Assume that $S\smallsetminus S_0$ is the union of two surfaces $S^+, S^-$ of class $C^2$ meeting along $S_0$. Let $\eta^+,\eta^-$ the restrictions of $\eta$ to $S^+$ and $S^-$, respectively. Then $S$ is area-stationary if and only if
\begin{enumerate}
\item $H_K=0$, and
\item $\eta^+-\eta^-$ is tangent to $S_0$.
\end{enumerate}
In particular, condition $H_K=0$ implies that $S\smallsetminus S_0$ is foliated by horizontal straight lines.
\end{corollary}

\begin{proof}
We may apply the divergence theorem to the second term in \eqref{eq:1stvar} to get
\[
\frac{d}{ds}\bigg|_{s=0} A_K(\varphi_s(S))=\int_{S\setminus S_0} H_Ku\,dS-\int_{S_0} u\,\escpr{\xi,(\eta^+-\eta^-)^\top}\,dS,
\]
where $\xi$ is the outer unit normal to $S^+$ along $S_0$. Hence the stationary condition is equivalent to $H=0$ on $S\smallsetminus S_0$ and $\escpr{\xi,\eta^+-\eta^-}=0$. The latter condition is equivalent to that $\eta^+-\eta^-$ be tangent to $S_0$.

That $H_K=0$ implies that $S\smallsetminus S_0$ is foliated by horizontal straight lines was proven in Theorem~3.14 in  \cite{PozueloRitore2021}.
\end{proof}

Since $\nu^+=J(Z^+), \nu^-=J(Z^-)$, where $Z^+$ and $Z^-$ are the extensions of the horizontal tangent vectors in $S^+,S^-$, we have that the second condition in Corollary~\ref{cor:minimal} is equivalent to
\begin{equation}
\label{eq:maincond}
\pi(J(Z^+))-\pi(J(Z^-)) \text{ is tangent to } S_0.
\end{equation}
So a natural question is, given a $C^2_+$ convex body $K$ containing $0$ in its interior, and a unit vector $v\in\sph^1$, can we find a pair of unit vectors $Z^+,Z^-$ such that \eqref{eq:maincond} is satisfied? If such vectors exist, how many pairs can we get? The answer follows from the next result.

\begin{lemma}
\label{lem:maincond}
Let $K$ be a convex body of class $C^2_+$ such that $0\in\intt(K)$. Given $v\in\rr^2\smallsetminus\{0\}$, let $L\subset\rr^2$ be the vector line generated~by $v$. Then, for any $u\in\ptl K$, we have the following possibilities
\begin{enumerate}
\item The only $w\in\ptl K$ such that $w-u\in L$ is $w=u$, or
\item There is only one $w\in\ptl K$, $w\neq u$ such that $w-u\in L$.
\end{enumerate}
The first case happens if and only if $L$ is parallel to the support line of $K$ at $u$.
\end{lemma}

\begin{proof}
Let $T$ be the translation in $\rr^2$ of vector $u$. Then $T(L)$ is a line that meets $\ptl K$ at $u$. The line $T(L)$ intersects $\ptl K$ only once when $L$ is the supporting line of $T(K)$ at $0$; otherwise $L$ intersects $\ptl K$ just at another point $w\neq u$ so that $w-u\in L$.
\end{proof}

\begin{remark}
\label{rk:Kcond}
We use Lemma~\ref{lem:maincond} to understand the behavior of characteristic curves meeting at a singular point $p\in S_0$. Let $Z^+, Z^-$ be the tangent vectors to the characteristic lines starting from $p$. Let $\nu^+,\nu^-$ be the vectors $J(Z^+),J(Z^-)$, and $L$ the line generated by the tangent vector to $S_0$ at $p$. The condition that $S$ is stationary implies that $\eta^+-\eta^-\in L$. If $w=\eta^+$ and $u=\eta^-$ are equal then $\nu^+=\nu^-$ are orthogonal to $L$, which implies that $Z^+,Z^-$ lie in $L$. This is not possible since characteristic lines meet tranversaly the singular line, again by Corollary~3.6 in \cite{MR2165405}.

Hence $\eta^+\neq\eta^-$ and $\eta^+$ is uniquely determined from $\eta^-$ by Lemma~\ref{lem:maincond}. Obviously the roles of $\eta^+$ and $\eta^-$ are interchangeable.
\end{remark}

\begin{figure}[h]
\begin{tikzpicture}[scale=0.7]

\draw[rotate=30,thick] (0,0) ellipse (3 and 2);
\draw (5,0) -- (-5,0) node[above]{$L$};
\draw (4, -1.73) -- (-6,-1.73)  node[above]{$T(L)=L+u$};

\draw[fill=black] (0,0) circle (0.1) node[above left]{$0$};;

\draw[fill=black] (-2.46,-1.73) circle (0.1);
\draw (-2.46,-1.73) node[xshift=20, yshift=8]{$w=\eta^+$};
\draw[->, rotate around={220:(-2.46,-1.73)}] (-2.46,-1.73) -- (-1.46,-1.73) node[xshift=-10]{$\nu^+$};
\draw[->, rotate around={130:(-2.46,-1.73)}] (-2.46,-1.73) -- (-1.46,-1.73) node[xshift=-10] {$Z^+$};

\draw[fill=black] (1,-1.73) circle (0.1);
\draw (1,-1.73) node[xshift=-12, yshift=8]{$u=\eta^-$};
\draw[->, rotate around={-60:(1,-1.73)}] (1,-1.73) -- (2,-1.73) node[xshift=12]{$\nu^-$};
\draw[->, rotate around={210:(1,-1.73)}] (1,-1.73) -- (2,-1.73) node[yshift=-10]{$Z^-$};

\draw (3.27,1.39) node{$\ptl K$};

\draw (-4, 2.32) -- (6,2.32);
\draw[fill=black] (0.85,2.32) circle (0.1) node [xshift=0, yshift=-12] {$u=w$};
\draw[->] (0.85,2.32) -- (.85,3.32) node[xshift=-25, yshift=-10]{$\nu^+=\nu^-$};
\end{tikzpicture}

\caption{Geometric construction to obtain $w=\eta^+$ from $u=\eta^-$ so that the stationary condition is satisfied. The case $\nu^+=\nu^-$ cannot hold.}
\end{figure}
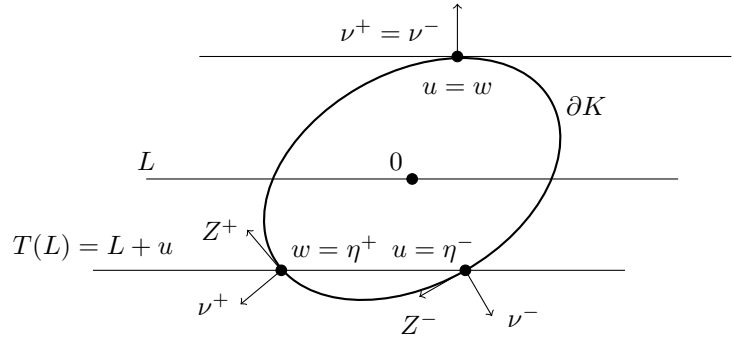

\section{Examples of entire $K$-perimeter minimizing horizontal graphs with one singular line}
\label{sec:one_singular_line}

Remark~\ref{rk:Kcond} implies that $Z^-$ can be uniquely determined from $Z^+$ when $S$ is a stationary surface. Let us see that this result can be refined to provide a smooth dependence of the oriented angle $\angle(v,Z^-)$ in terms of $\angle(v,Z^+)$. We use complex notation for horizontal vectors assuming that the horizontal distribution is positively oriented by $v,J(v)$ for any $v\in \mathcal{H}\smallsetminus\{0\}$.

\begin{lemma}\label{lem:ang}
Let $K$ be a convex body of class $C^2_+$ with $0\in\intt(K)$. Consider a unit vector $v\in\rr^2$ and let $L\subset\rr^2$ be the vector line generated~by $v$. Then, for any $\alpha\in (0,\pi)$ there exists a unique $\beta\in (\pi,2\pi)$ such that if $Z^+=v e^{i \alpha}$, $Z^-=v e^{i  \beta}$, then $\pi(J(Z^+))- \pi(J(Z^+))$ belongs to $L$.

Moreover the function $\beta: (0,\pi)\to (\pi,2\pi)$ is of class $C^1$ with negative derivative. 
\end{lemma}

\pagebreak

\begin{proof}
We change coordinates so that $L$ is the line $y=0$. We observe that $Z^+=ve^{i\alpha}$ implies that $J(Z^+)=ve^{i(\alpha+\pi/2)}$. We define $(x,y):\sph^1\to\ptl K$ by
\[
(x(\alpha),y(\alpha))=N_K^{-1}(ve^{i(\alpha+\pi/2)}),
\]
where $N_K:\ptl K\to\sph^1$ is the (outer) Gauss map of $\ptl K$. The functions $x,y$ are $C^1$ since $N_K$ is $C^1$. The point $(x(\alpha),y(\alpha))$ is the only one in $\ptl K$ such that the clockwise oriented tangent vector to $\ptl K$ makes an angle $\alpha$ with the positive direction of the line $L$. A line parallel to $L$ meets $\ptl K$ at a single point only when $\alpha+\pi/2=\pi/2$ or $\alpha+\pi/2=3\pi/2$. Hence, for $\alpha\in (0,\pi)$, there is a unique $\beta\in (\pi,2\pi)$ such that 
\[
(x(\beta),y(\beta))-(x(\alpha),y(\alpha))\in L.
\]
Observe that, for $\alpha\in (0,\pi)$, we have $dy/d\alpha>0$ and, for $\beta\in (\pi,2\pi)$, we get $dy/d\beta<0$. We can use the implicit function theorem (applied to $y( \beta)-y(\alpha)$) to conclude that $\beta$ is a $C^1$ function of $\alpha$. Moreover
\[
\frac{d\beta}{d\alpha}=\frac{dy/d\alpha}{dy/d\beta}<0,
\]
as desired.
\end{proof}

Now we give the main construction in this section.

We fix a vector $v\in\rr^2\smallsetminus\{0\}$ and the line $L_v=\{\lambda v: \lambda\in \rr\}$. For every $\lambda\in\rr$, we consider two half-lines, $r_\lambda^+,r_\lambda^-\subset\rr^2$, extending from the point $p=\lambda v\in L_v$ with angles $\alpha(\lambda)$ and $\beta(\lambda)$ respectively. Here $\alpha:\rr\to (0,\pi)$ is a non-increasing function and $\beta(\la)$ is the composition of $\alpha(\lambda)$ with the function obtained in Lemma~\ref{lem:ang}. Hence $\beta(\lambda)$ is a non-decreasing function. The line $L_v$ can be lifted to the horizontal straight line $R_v=L_v\times\{0\}\subset\hh^1$ passing through the point $(0,0,0)$, and the half-lines $r^\pm_\lambda$ can be lifted to horizontal half-lines $R^\pm_\lambda$ starting from  the point $(\lambda v,0)$ in the line $R_v$.

The surface obtained as the union of the half-lines $R_\lambda^+$ and $R_\lambda^-$, for $\lambda\in\rr$,  is denoted by $\Sigma_{v,\alpha}$. Since any $R^\pm_\lambda$ is a graph over $r_\lambda^\pm$ and $\bigcup_{\lambda\in\rr}(r_\lambda^+\cup r_\lambda^-)$ covers the $xy$-plane, we can write the surface $\Sigma_{v,\alpha}$ as the graph of a continuous function $u_\alpha:\rr^2\to\rr$.
Writing $v=e^{i\alpha_0}$, the surface $\Sigma_{v,\alpha}$ can be parametrized by $\Psi:\rr^2\to\rr^3$ as follows
\begin{equation}
\label{eq:param}
\Psi(\lambda,\mu)= \begin{cases}
\big(\lambda e^{i\alpha_0}+\mu  e^{i(\alpha_0+\alpha(\lambda))},- \mu\lambda \sin \alpha(\lambda) \big),  & \mu\geq 0, \\
 \big(\lambda e^{i \alpha_0}+|\mu|  e^{i(\alpha_0+\beta(\lambda))},- |\mu|\lambda \sin \beta(\lambda )\big), & \mu\leq 0.
\end{cases}
\end{equation}

\pagebreak

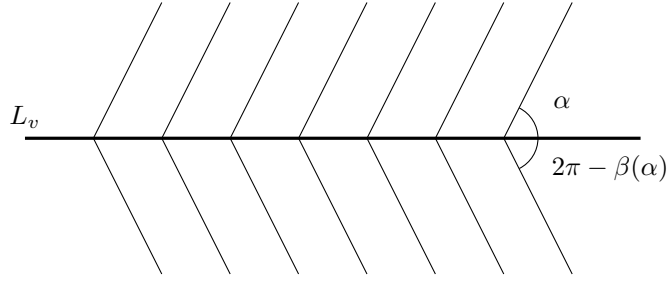
\begin{figure}[h]

\centering

\begin{tikzpicture}[scale=0.9]

\draw[very thick] (5,0) -- (-4,0) node[above]{$L_v$};

\draw (-3,0) -- (-2,2);
\draw (-2,0) -- (-1,2);
\draw (-1,0) -- (0,2);
\draw (0,0) -- (1,2);
\draw (1,0) -- (2,2);
\draw (2,0) -- (3,2);
\draw (3,0) -- (4,2);
\draw (3.5,0) arc [start angle=0, end angle=63, radius=0.5] node[xshift=16,yshift=2]{$\alpha$};
\draw (3.5,0) arc [start angle=0, end angle=-63, radius=0.5] node[xshift=34,yshift=0]{$2\pi-\beta(\alpha)$};

\draw (-3,0) -- (-2,-2);
\draw (-2,0) -- (-1,-2);
\draw (-1,0) -- (0,-2);
\draw (0,0) -- (1,-2);
\draw (1,0) -- (2,-2);
\draw (2,0) -- (3,-2);
\draw (3,0) -- (4,-2);

\end{tikzpicture}

\caption{The planar configuration to obtain the surface $\Sigma_{v,\alpha}$. Here $\alpha$ is a constant function and $K$ is the unit disk $D$. Such surfaces were called \emph{herringbone surfaces} by Young \cite{Young2020} as they are the union of horizontal rays that branch out of a horizontal line.}

\end{figure}

\begin{example}
A special example to be considered is the sub-Riemannian cone $\Sigma_\alpha$, where $\alpha\in (0,\pi)$. The projection of $\Sigma_\alpha$ to the horizontal plane $t=0$ is composed of the line $y=0$ and the half-lines starting from points in $y=0$ with angles $\alpha$ and $-\alpha$. This cone can be parametrized, for $s\in\rr, t\ge 0$, by
\begin{equation*}
(u,v)\mapsto (u+v\cos\alpha,v\sin\alpha,-uv\sin\alpha)
\end{equation*}
when $y\ge 0$, and by
\begin{equation*}
(u+v\cos\alpha,-v\sin\alpha,uv\sin\alpha)
\end{equation*}
when $y\le 0$. A straigthforward computation implies that $\Sigma_\alpha$ is the $t$-graph of the function
\begin{equation}
\label{eq:ualpha}
u_\alpha(x,y)=-xy+\cot\alpha \, y|y|.
\end{equation}

Observe that
\begin{equation}
\lim_{\alpha\to 0} u_\alpha(x,y)=\begin{cases} +\infty, & y>0, \\
0, &y=0, \\
-\infty, & y<0,
\end{cases}
\end{equation}
so that the subgraph of $\Sigma_\alpha$ converges pointwise locally when $\alpha\to 0$ to a vertical half-space.
\end{example}

 The following rsult provides some properties of $u_\alpha$ when $\alpha(\la)$ is a smooth~function of $\la$.
\begin{proposition}
\label{prop:one_singular_line}
Let $\alpha\in C^k(\rr)$, $k\geq2$, be a non-decreasing function. Then
\begin{enumerate}
\item[\emph{i)}] $u_\alpha$ is a $C^k$ function in $\rr^2\setminus L_v$,
\item[\emph{ii)}] $u_\alpha$ is merely $C^{1,1}$ near $L_v$ when $\beta\neq \alpha+\pi$.
\item[\emph{iii)}] $u_\alpha$ is $C^\infty$ in any open set $I$ of values of $\lambda$ when $\beta= \alpha+\pi$ on $I$.
\item[\emph{iv)}] 
$\Sigma_{v,\alpha}$ is $K$-perimeter-minimizing when $\beta= \beta(\alpha)$.
\item[\emph{v)}] The projection of the singular set of $\Sigma_{v,\alpha}$ to the $xy$-plane is $L_v$.
\end{enumerate}
\end{proposition}

\pagebreak

\begin{proof}
i), ii), iii) and v) are proven in Lemma 3.1 in \cite{MR2448649}.

We prove iv) by a calibration argument. We shall drop the subscript $\alpha$ to simplify the notation. Let $E$ be the subgraph of $u$ and $F\subseteq\hh^1$ such that $F=E$ outside a Euclidean ball centered at the origin. Let $P=\{(z,t) : \escpr{z,v}=0\}$, $P^1=\{(z,t) : \escpr{  z,v}>0\}$ and $P^2=\{(z,t) : \escpr{  z,v}<0\}$. We define two vector fields $U^1$, $U^2$ on $P^1$, $P^2$ respectively by vertical translations of the vectors $\pi(\nu_E)|_{P^1}=\eta^+$ and $\pi(\nu_E)|_{P^2}=\eta^-$. They are $C^2$ in the interior of the half-spaces and extend continuously to the boundary plane $P$. As $\divv (U^j)_{(z,t)}$ coincides with the sub-Finsler mean curvature of the translation of $\Sigma_{v,\alpha}$ passing through $(z,t)$ as defined in \cite{PozueloRitore2021}, and these surfaces are foliated by horizontal straight lines in the interior of the half-spaces, by Theorem 3.14 in \cite{PozueloRitore2021} we get
\[
\divv U^j=0 \quad j=1,2.
\]
Here $\divv U$ is the Riemannian divergence of the vector field $U$. 
We apply the divergence theorem (Theorem 2.1 in \cite{MR2448649}) to get
\begin{equation*}
0=\int_{F\cap P^j \cap B}\divv U^j=\int_F\escpr{U^j,\nu_{P^j\cap B}}|\partial (P^j\cap B)|+\int_ {P^j\cap B}\escpr{U^j,\nu_F}|\partial F|.
\end{equation*}
Let $C=P\cap\bar{B}$. Then, for every $p\in C$, we have $\nu_{P^1\cap B}=J(v)$ is a normal vector to the plane $P$ and $\nu_{P^2\cap B}=-J(v)$, $U^1=\eta^+$ and $U^2=\eta^-$. 
Hence, by Lemma \ref{lem:ang}, we get
\[
\escpr{U^1,\nu_{P^1\cap B}}+\escpr{U^2,\nu_{P^2\cap B}}=\escpr{\eta^+-\eta^-,J(v)}=0 \quad p\in C.
\] 
Adding the above integrals we obtain
\begin{equation}\label{eq:2}
0=\sum_{j=1,2}\int_F\escpr{U^j,\nu_B}d|\partial B|+\int_{B\cap\text{int}(H^j)}\escpr{U^j,\nu_F} d|\partial F|.
\end{equation}

From the Cauchy-Schwarz inequality and the fact that $|\partial F|$ is a positive measure, we get that
\begin{equation}\label{eq:3}
\sum_{j=1,2}\int_{B\cap P^j}\escpr{U^j,\nu_F} d|\partial F|\leq P_K( F,B).
\end{equation}
In particular, if we apply the same reasoning to $E$, equality holds and
\begin{equation}\label{eq:4}
0=\sum_{j=1,2}\int_E\escpr{U^j,\nu_B}d|\partial B|+P_K(E,B).
\end{equation}
From \eqref{eq:2}, \eqref{eq:3}, \eqref{eq:4} and the fact that $F=E$ in the boundary of $B$, we get
\[
P_K(E,B)\leq P_K(F,B),
\]
as desired.
\end{proof}

The general properties of $\Sigma_{v,\alpha}$ when $\alpha$ is only continuous are given in the following result.
\begin{proposition}
\label{prop:cont}
Let $\alpha:\rr\to\rr$ be a continuous and non-decreasing function.
Then
\begin{enumerate}
\item[\emph{i)}] $u_\alpha$ is locally Lipschitz in Euclidean sense,
\item[\emph{ii)}] $E_\alpha$ is a set of locally finite perimeter in $\hh^1$, and 
\item[\emph{iii)}] $\Sigma_{v,\alpha}$ is $K$-perimeter-minimizing in $\hh^1$.
\end{enumerate}
\end{proposition}

\begin{proof}
i) and ii) are proven in \cite{MR2448649}, Proposition 3.2. Let 
\[
\alpha_\varepsilon (x) = \int_\rr \alpha(y)\delta_\varepsilon (x-y) dy
\]
the usual convolution, where $\delta$ is a Dirac function and $\delta_\varepsilon=\frac{\delta(x/\varepsilon)}{\varepsilon}$. Then $\alpha_\varepsilon$ is a $C^\infty$ non-decreasing function and $\alpha_\varepsilon$ converges uniformly to $\alpha$ on compact sets of $\rr$. By Lemma \ref{lem:ang}, $\beta_\varepsilon=\beta(\alpha_\varepsilon)$ is a $C^1$ non-decreasing function. Since $\beta$ is $C^1$ with respect to $\alpha$ it follows the  uniform convergence on compact sets of $\beta_\varepsilon$ to a function $\bar{\beta}$.

Take $F\subset \hh^1$ so that $F=E$ outside a Euclidean ball centered at the origin. We follow the arguments of the proof of iv) in Proposition \ref{prop:one_singular_line} and define vector fields $\divv(U^j_\varepsilon)$ translating vertically $\pi(\nu_{E_\varepsilon})$, where $E_\varepsilon$ is the subgraph of $\Sigma_{\alpha_\varepsilon}$, to obtain by the divergence theorem
\[
\sum_{j=1,2}\int_{B\cap \text{int}(P^i)}\escpr{U_\varepsilon^j,\nu_{E_\varepsilon}}|\partial E_\varepsilon|=\sum_{j=1,2}\int_{B\cap \text{int}(P^i)}\escpr{U_\varepsilon^j,\nu_{F}}|\partial F|,
\]
the left hand side is the $K$-perimeter of $E_\varepsilon$, while the right hand side is trivially bounded by the $K$-perimeter of $F$. Therefore
\[
P_K( E_\varepsilon,B)\leq P_K( F,B).
\]
Since $E_\varepsilon$ converges uniformly in compact sets to $E$, we obtain the result. 
\end{proof}

We study now with some detail the case when $\Sigma_{v,\alpha}$ is a $C^\infty$ surface.

\begin{corollary}
\label{cor:cones}
When $\alpha$ is constant, the surface $\Sigma_{v,\alpha}$ is a $K$-perimeter-minimizing cone in $\hh^1$ of class $C^{1,1}$. The singular set is a horizontal straight line and the regular part of $\Sigma_{v,\alpha}$ is a $C^\infty$ surface.
\end{corollary}

The following extends the already known result that in the sub-Riemannian setting the surfaces $\Sigma_{v,\pi/2}$ are $C^\infty$.
\begin{lemma}\label{lem:condreg}
Let $v\in\rr^2 \smallsetminus \{0\}$ and $\alpha\in (0,\pi)$ be fixed. If $K$ is centrally symmetric with respect to $O=\frac{1}{2}\eta^++\frac{1}{2}\eta^-$ then  $\beta(\alpha)=\alpha+\pi$, where $\eta^+=\pi(J(v e^{i\alpha}))$ and $\eta^-=\pi(J(v e^{i\beta}))$.
\end{lemma}
\begin{proof}
Let $K$ be centrally symmetric with respect to $O$. Then $\eta^-$ is the symmetric point of $\eta^+$. On the other hand, the convex body $K-O$ is symmetric with respect to the origin. Then the dual norm is even and, in particular,  $\pi_{K-O}(-\nu^+)=-\pi_{K-O}(\nu^+)$. Now, since a translation takes symmetric points of $K-O$ with respect to the origin to symmetric points of $K$ with respect to $O$,  
we get $\nu^-=-\nu^+$. This implies that $\beta(\alpha)=\alpha+\pi$.
\end{proof}

The existence of a convex body $K$ of class $C^2_+$ such that $0\in \intt(K)$ for which $\Sigma_{v,\alpha}$ is $C^\infty$ is studied in Corollary \ref{cor:reg} and Proposition \ref{prop:reg}.

\begin{corollary}\label{cor:reg}
Let $v\in\rr^2 \smallsetminus \{0\}$ and $\alpha\in (0,\pi)$ be fixed. Then there exists a convex body $K$ of class $C^2_+$ with $0\in \intt(K)$ such that $\Sigma_{v,\alpha}$ is $C^\infty$.
\end{corollary}
\begin{proof}
To construct the convex body $K$, fix a point $p\in\{(x,y) : \escpr{(x,y),ve^{i\alpha}}>0\}$ and $O\in J(L)+p\cap L$, where $L$ is the vector line generated~by $v$. Then any $K$ of class $C^2_+$ centrally symmetric with respect to $O$ containing the origin such that $p\in\partial K$ and $ve^{i\alpha}\bot T_p\partial K$ satisfies the hypothesis of Lemma \ref{lem:condreg}, where $\eta^+=p$ and $\eta^-$ is the symmetric of $\eta^+$ with respect to $O$. Thus, by $(iii)$ in Proposition \ref{prop:one_singular_line} we get that $\Sigma_{v,\alpha}$ is $C^\infty$.
\end{proof}
\begin{proposition}\label{prop:reg}
Given a convex body $K$ of class $C^2_+$ with $0\in \intt(K)$, there exists $v\in\rr^2$ such that $\Sigma_{v,\pi/2}$ is $C^{\infty}$.
\end{proposition}
\begin{proof}
Let $p$ and $q$ be points in $K$ at maximal distance. Then the lines through $p$ and $q$ orthogonal to $q-p$ are support lines to $K$. Taking $v=q-p$ and setting  $p=\eta^+$ we have $q=\eta^-$, while the vectors $\nu^+$ and $\nu^-$ are over the line $L(v)$, that is, $Z^+$ $Z^-$ make angles $\pi/2$ and $3\pi/2$ with $L(v)$.
\end{proof}

For fixed $v\in\rr^2 $, we define the surface $\Sigma_{v,\alpha}^+$ as the one composed of all the horizontal half-lines $R_\lambda^+$ and $R_\lambda^-\subseteq\rr^2$ extending from the lifting of the point $p=\lambda v\in L_v$, $\lambda\ge 0$, to $\hh^1$. The surface $\Sigma_{v,\alpha}^+$ has a boundary composed of two horizontal lines and its singular set is the ray $L_v^+=\{\lambda v : \lambda > 0 \}$. We present some pictures of such surfaces.

\begin{figure}[H]
\centerline{\includegraphics[width=0.20\textwidth]{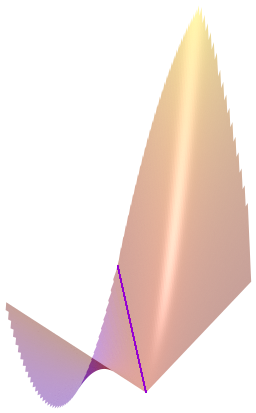}}
\caption{The surface $\Sigma_{\pi/3,\pi/6}^+$ associated to the norm $\norm{\cdot}_D$, where $D$ is the unit disk. The singular~set corresponds to the purple ray of angle $e^{i \pi/3}$.}
\label{fig:1}
\end{figure}

\pagebreak

\begin{figure}[H]
\centerline{\includegraphics[width=0.20\textwidth]{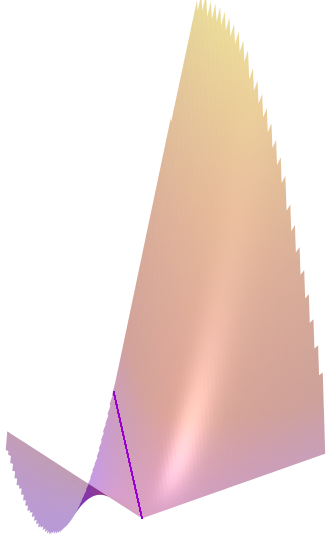}}
\caption{The surface $\Sigma_{\pi/3,\pi/6}^+$ associated to the $p$-norm with $p=1.5$. The left part of the figure coincides with the left part of Figure \ref{fig:1}, while the angle $\beta$ is bigger. Notice that also the height has increased.}
\end{figure}

\begin{figure}[H]
\centerline{\includegraphics[width=0.20\textwidth]{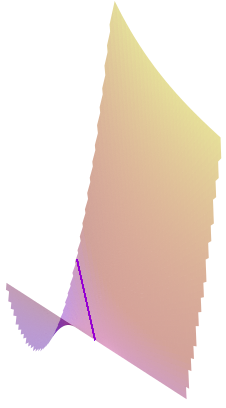}}
\caption{The surface $\Sigma_{\pi/3,\pi/6}^+$ with $\beta=\alpha+\pi$. There existence of $K$ is granted by Corollary \ref{cor:reg}.}
\end{figure}

\section{Area-Minimizing Cones in $\hh^1$}
\label{sec:cones}

We proceed now to construct examples of $K$-perimeter minimizing cones in $\hh^1$ with an arbitrary finite number of horizontal half-lines meeting at the origin. The building blocks for this construction are liftings of circular sectors of the cones considered in Corollary~\ref{cor:cones}.

We first prove the following result.

\begin{lemma}
\label{lem:sector}
Let $K$ be a convex body of class $C^2_+$ such that $0\in\intt(K)$. Let $u, w \in \mathbb{S}^1$,  $\theta=\angle(u,w)>0$. Then there exists $v \in \mathbb{S}^1$ such that the vector line $L_{v}$ generated~by $v$ splits the sector determined by $u$ and $w$ into two sectors of oriented angles $\alpha$ and $\beta$ such that $\alpha+\beta=\theta$. Moreover, the stationary condition $\pi(J(u))-\pi_K(J(w)) \in L_{v} $  is satisfied.
\end{lemma}

\begin{proof}
Let $\nu_u = J(u)$, $\nu_w= J(w)$ and $\eta_u =\pi(\nu_u)$, $ \eta_w=\pi(\nu_w)$, $\eta_u \ne \eta_w$ since $\pi$ is a $C^1$ diffeomorphism. Thus there exists a unique line $\tilde{L}$ passing through $\eta_u$ and $\eta_w$ and $L=\tilde{L}-\eta_u$ is a straight line passing though the origin. Notice that $\tilde{L}$ splits $\partial K$ in two connect open components $\partial K_1$ and $\partial K_2$. There exist two points $\eta_1 \in \partial K_1$ and $\eta_2 \in \partial K_2 $ such that $L+\eta_1$ (resp. $L+\eta_2$) is the support line at $\eta_1$ (resp. $\eta_2$). Setting $v_1= N_{\partial K}( \eta_1)$ and $v_2= N_{\partial K}( \eta_2)$ we gain that $v_i$ for $i=1,2$ is perpendicular to $L$. Without loss of generality we set that $-J(v_1)$ belongs to the portion of plane identified by the $\theta$ and $-J(v_2)$  belongs   to the portion of plane identified by the  $2 \pi - \theta$. Then we  set $v=-J(v_1)$.  Notice that $v$ splits $\theta$ in two angles $\beta=\angle(u,v)$, $\alpha=\angle(v,w)$ with  $\theta= \alpha+ \beta$ and $L=L_v$.
\end{proof}

Now we proceed with the construction inspired by the sub-Riemannian construction in \cite{MR4307010}. For $k\ge 3$ consider a fixed angle $\theta_0$ and family of positive oriented angles $\theta_1,\ldots,\theta_k$ such that $\theta_1+\cdots +\theta_k=2\pi$. Consider the planar vectors $u_0=(\cos(\theta_0),\sin(\theta_0))$ and
\[
u_i=(\cos(\theta_0+\theta_1+\cdots +\theta_i),\sin(\theta_0+\theta_1+\cdots +\theta_i)), \qquad i=1,\ldots,k.
\]
Observe that $u_k=u_0$. For every $i\in\{1,\ldots,k\}$ consider the vectors $u_{i-1},u_i$ and apply Lemma~\ref{lem:sector} to obtain a family of $k$ vectors $v_i$ in $\sph^1$ between $u_{i-1}$ and $u_i$. We lift the half-lines $L_i=\{\lambda v_i:\lambda\ge 0\}$ to horizontal straight lines passing through $(0,0,0)\in\hh^1$, and we also lift the half-lines
\[
\lambda v_i+\{\rho u_{i-1}:\rho\ge 0\}, \qquad \lambda v_i+\{\rho u_i:\rho\ge 0\},
\]
to horizontal straight lines starting from $(\lambda v_i,0)$. This way we obtain a surface \[
C_K(\theta_0,\theta_1,\ldots,\theta_k)
\]
with the following properties

\begin{theorem}
\label{thm:cones}
The surface $C_K(\theta_0,\theta_1,\ldots,\theta_k)$ is $K$-perimeter-minimizing cone which is the graph of a $C^1$ function.
\end{theorem}

\begin{proof}
$C_K(\theta_0,\theta_1,\ldots,\theta_k)$ is a cone by construction. It is an entire graph since it is composed of horizontal lifting of straight half-lines in the $xy$-plane that covered the whole plane without intersecting themselves transversally. The $K$-perimeter-minimizing property follows in a similar way to from Proposition~2.4 in \cite{MR4307010}. That it is the graph of a $C^1$ function is proven like in Proposition~3.2(4) in \cite{MR4307010}.
\end{proof}

A particular example of area-minimizing cones are those who uses the sub-Riemannian cones $C_\alpha$ restricted to the circular sector with $\theta\in(-\alpha,\alpha)$ as as model piece of the cone. Taking $K=D$, $k\geq 3$, and the angle $\alpha=\pi/k$, we define
\[
C(k)=C_D\big(\frac{\pi}{k},\frac{2\pi}{k},\ldots,\frac{2\pi}{k}\big).
\]
Let us denote by $u_k$ the functions in $\rr^2$ whose graph is $C(k)$. The behavior when $k$ tends to infinity of $u_k$ in a disk is analyzed in the following result.

\begin{proposition}
 The sequence $u_k$ converge to $0$ uniformly on compact subsets of $\rr^2$. Moreover, the sub-Riemannian area of $u_k$ converges locally to the sub-Riemannian area of the plane $t=0$. Moreover the sub-Riemannian area of $u_k$ converges to the one of the plane $t=0$.
\end{proposition}

\begin{proof}
Since $u_k$ is obtained by collating some rotated copies of $u_\alpha$, where $\alpha=\pi/k$, we can estimate the height of $u_k$ by the height of $u_\alpha$.  By \eqref{eq:ualpha},  using polar coordinates $(r,\theta)$, where $\theta\in [-\alpha,\alpha]$ and $r<r_0$, we get
\begin{equation*}
|u_\alpha|\le 2r_0^2 |\sin(\pi/k)|
\end{equation*}
on $D(r_0)=\overline{B}(0,r_0)$. The claim follows since $\lim_{k\to\infty}\sin(\pi/k)=0$.

The sub-Riemannian area of the graph of $u_k$ over $D(r_0)$ is given by
\[
A_D(u_k,r_0)=\int_{D(r_0)}\| \nabla u_k + (-y,x)\| dxdy.
\]
Since the sub-Riemannian perimeter is rotationally invariant, we can decompose the above integral as $k$ times the area of the cone $C_\alpha$ in the circular sector with $\theta\in(-\alpha,\alpha)$ and $r<r_0$.  By \eqref{eq:ualpha}, it is immediate that 
\[
\|\nabla u_k(x,y)+(-y,x)\|=2|y|\sin^{-1}(\alpha).
\]
A direct computation shows that
\[
A_D(u_k,r_0)=\frac{4\pi r_0^3}{3} \frac{1-\cos \pi/k}{(\pi/k)\sin\pi/k}.
\]
Then $A_D(u_k,r_0)$ tends to $\frac{2\pi r_0^3}{3}$ as $k\to +\infty$. 
\end{proof}

\begin{figure}[htp]
\centerline{\includegraphics[width=0.33\textwidth]{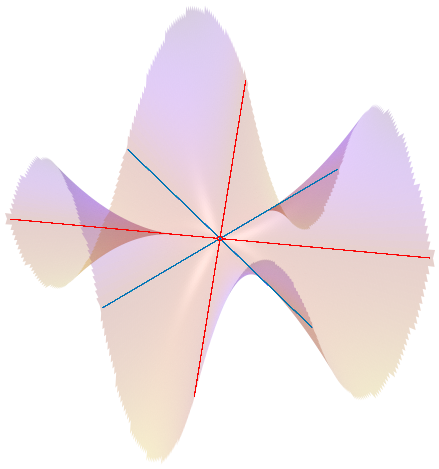}}
\caption{The cone  $C(4)$.
 The singular set is  composed of the red rays of angle $0,\pi/2,\pi,(3\pi)/2$, while the rays of angles $\pi/4,(3\pi)/4,(5\pi)/4),(7\pi)/4$, where two pieces of the construction meet, are depicted in cyan.}
\end{figure}

\begin{figure}[htp]
\centerline{\includegraphics[width=0.33\textwidth]{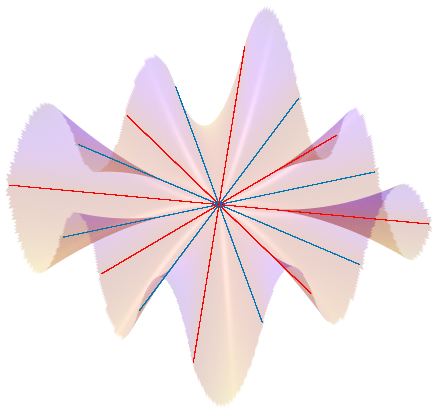}
\hspace{1cm}
\includegraphics[width=0.4\textwidth]{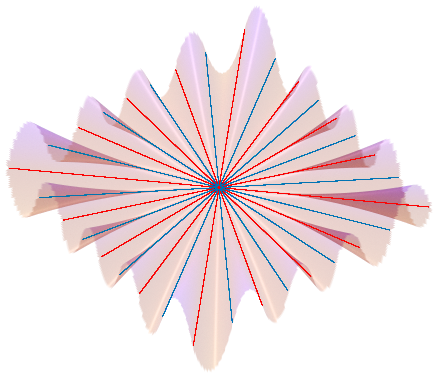}}
\caption{The cones $C(8)$ and $C(16)$. They are depicted at the same in this Figure and the previous one. As the number of angles increases, the cone produces more oscilations of smaller height.
}
\end{figure}

\end{document}